\documentclass[preprint,12pt]{elsarticle}
\usepackage[colorlinks]{hyperref}
\usepackage[usenames,dvipsnames]{color}

\newtheorem{thm}{Theorem}

\newtheorem{rem}{Remark}
\newtheorem{coro}{Corollary}
\newtheorem{ex}{Example}

\newcommand{\F}{\mathbb{F}}
\newcommand{\Z}{\mathbb{Z}}

\newenvironment{proof}{\noindent\textbf{Proof.}\quad}
{\hspace{\stretch{1}}%
\rule{1ex}{1ex}\\}

\usepackage{amssymb}
\usepackage{amsmath}

\journal{}

\begin{document}

\begin{frontmatter}


\tnotetext[label1]{The research is supported by the National Natural Science Foundation of China under Grant 12471490.}
\tnotetext[label2]{Minjia Shi and Lu Wang  are with the Key Laboratory of Intelligent Computing Signal
		Processing, Ministry of Education, School of Mathematical Sciences, Anhui
		University, Patrick Sol\'e is with I2M (Aix Marseille Univ, CNRS).}
\author{Minjia Shi\fnref{1}}
\ead{smjwcl.good@163.com}
\author{Lu Wang\fnref{1}}
\ead{wanglu5666@163.com}
\author{Patrick Sol\'e\fnref{2}}
\ead{ sole@enst.fr}


\title{The log-concavity of two graphical sequences\thanks{The research is supported by the National Natural Science Foundation of China under Grant 12471490.}}
\affiliation[1]{organization={School of Mathematical Sciences, Anhui
		University},
            city={Hefei},
             postcode={230601},
            state={Key Laboratory of Integrated Service Networks, Xidian University, Xi'an,
		710071},
            country={China}}

\affiliation[2]{organization={I2M (Aix Marseille Univ, CNRS)},
            city={Marseilles},
            country={France}}

\begin{abstract}
We show that the large Cartesian powers of any graph have log-concave valencies with respect to a fixed vertex.
The series of  valencies of distance regular graphs is well-known to be log-concave. Consequences for strongly regular graphs, two-weight codes, and completely regular codes are explored. By P-Q duality of association schemes the series of multiplicities of $Q$-polynomial association schemes is shown, under some assumption on the Krein parameters, to be log-concave.
\end{abstract}



\begin{keyword}
log-concavity \sep distance regular graphs \sep association schemes \sep valencies \sep multiplicities

\MSC[2020] Primary 05 E 30, 05 C 30 ; Secondary 94 B 25


\end{keyword}

\end{frontmatter}



\section{Introduction}
\label{sec1}
Recall that a sequence $s_i$ of integers is {\em log-concave} if $ s_i^2\ge s_{i-1}s_{i+1}$ for all $i>1.$
Many classical sequences in enumerative combinatorics are log-concave \cite{Sta}. Log-concavity implies {\em unimodality} (i.e. the existence of an index $i$ before which $s_i$ is nondecreasing and after which it is nonincreasing) but the converse is not true. In this note, we will consider two sequences that occur naturally when studying a graph:
The multiplicities of its eigenvalues, and the number of vertices at a given distance from a fixed vertex. These two sequences are not interesting for general graphs. As a recent result of Tao and Vu  shows the spectrum of most large graphs is simple \cite{TV}. Of course, the constant sequence equal to one is log-concave. As for the second sequence some examples that are not log-concave will be given in Section \ref{sec3}. In particular, an example of a distance degree regular with non log-concave series of valencies is given there.

On the positive side we give a Cartesian product construction of a special class of distance degree regular graphs with a log-concave series of valencies. A special subclass of distance degree regular graphs is that of distance regular graphs \cite{BI,BCN}. In that class of graphs, the second sequence is that of valencies. This sequence has been known to be unimodal for a long time \cite{ TL}. 
The series of valencies is known  to be log-concave by \cite[p.205]{BI}. As a corollary we derive some bounds on the size of strongly regular graphs that seem to be new \cite{BM}. We also give a bound on the weight distributions of projective two-weight codes that does not hold for non-projective two-weight codes. By using the duality between $P$ and $Q$ association schemes we show that, under some assumption on the Krein array, the multiplicities of $Q$-polynomial association schemes form a log-concave sequence. Inasmuch as these multiplicities are dimensions of eigenspaces, this latter result concerns the first sequence mentioned above.

The material is arranged as follows. The next section collects the basic notions and notations needed for the rest of the paper. Section \ref{sec3} gives some examples of non log-concave sequences in general graphs. Section \ref{sec4}
contains the main results on distance regular graphs, strongly regular graphs and completely regular codes. Section \ref{sec5} concludes the article.


\section{Preliminary}
\label{sec2}
\subsection{Graphs}
\label{sec2.1}

All the graphs in this note are simple, connected, without loops or multiple edges.
The neighborhood $\Gamma(x)$ is the set of vertices connected to $x$.

The {\em degree} of a vertex $x$ is the size of $\Gamma(x)$.

A graph is {\em regular} if every vertex has the same degree. The $i$-neighborhood $\Gamma_i(x)$ is the set of vertices at geodetic distance $i$ to $x$.

The \emph{diameter} $d$ of the graph is the maximum $i$ such that for some vertex $x$ the set $\Gamma_i(x)$ is nonempty.

A graph is {\em distance degree regular} (DDR) if  $|\Gamma_i(x)|$ is independent of $x,$ for all $0\le i\le d.$

A graph is {\em distance regular} (DR) if for every two vertices $u$ and $v$ at distance $i$ from each other the values $b_i=| \Gamma_{i+1}(u)\cap \Gamma(v)|$, $c_i=| \Gamma_{i-1}(u)\cap \Gamma(v)|$ depend only on $i$ and do not depend on the choice of $u$ and $v$. It can be seen that every DR graph is DDR but not conversely.

In $\Gamma$ is DDR, the graphs $\Gamma_i$ are regular of degree $v_i$ and we will refer to the $v_i$s as the {\em valencies} of $\Gamma$; in the case that $\Gamma$ is DR, the sequence $\{b_0,\ldots, b_{\mathrm{d}-1}; c_1,\ldots, c_{\mathrm{d}}\}$ is called the \emph{intersection array} of the distance-regular graph.

A distance-regular graph of diameter $2$ is called a {\em strongly regular graph}. The parameters of a strongly regular graph are usually listed as the quadruple $(\nu,\kappa,\lambda,\mu)$, where $\nu$ is the number of vertices, $\kappa$ the degree, $\lambda$ is the number of common neighbors of a pair of connected vertices, $\mu$ the number of common neighbors of a pair of disconnected vertices.

The {\em spectrum} of a graph is the set of distinct eigenvalues of its adjacency matrix. It is denoted by $\{\theta_0^{m_0},\theta_1^{m_1},\ldots\}$, where $m_i$ stands for the multiplicity of the eigenvalue $\theta_i$.

The {\em Cartesian product $G\square H$} of graphs $G$ and $H$ is a graph such that:
the vertex set of $G\square H$ is the Cartesian product $V(G) \times V(H)$, and two vertices $(u,v)$ and $(u',v')$ are adjacent in $G\square H$ if and only if either $u = u'$ and $v$ is adjacent to $v'$ in $H$, or $v = v'$ and $u$ is adjacent to $u'$ in $G$. 

The {\em Cartesian power $G^{\square n} $} of a graph $G$ is defined by induction on the integer $n$ as $G^{\square n}=G^{\square {n-1}}\square G$ from $G^{\square 1}=G. $

We use the following notation for special graphs:\par

 $K_n$ for the complete graph on $n$ vertices;\par
$K_{n\times m}$ for the bipartite complete graph on $n\times m$ vertices;\par
$C_q$ for the cycle graph on $q$ vertices.

\subsection{Codes}\label{sec2.2}
A linear code of length $n$, dimension $k$, minimum distance $d$ over the finite field $\F_q,$ is called an  $[n,k]_q$ code. Its elements are called  {\em codewords}. The  {\em weight} of a codeword is the Hamming weight. The {\em weights} of the code are the weights achieved by its codewords.

A {\em two-weight code} of parameters $[n,k;w_1,w_2]_q$ is an $[n,k]_q$ linear code having two nonzero weights $w_1$ and $w_2$.

The duality is understood with respect to the standard inner product.

A {\em coset} of a linear code $C$ is any translate of $C$ by a constant vector.

A {\em coset leader} is any coset element that minimizes the weight.

The {\em weight of a coset} is the weight of any of its leaders.

The {\em coset graph} $\Gamma_C$ of a code $C$ is defined on the cosets of $C$, two cosets being connected if they differ by a coset of weight one.

\subsection{Association schemes}\label{sec2.3}
Let $X$ be a finite set, and $R=\{R_0,\dots,R_{d}\}$ denote a partition of $X\times X$ into $d+1$ relations. Write $A_i$ for the adjacency matrix of the graph $(X,R_i).$ The pair $(X,R)$ forms a symmetric association scheme with $d$ classes if the following three conditions hold
\begin{enumerate}

\item [(1)] $A_0=I$ the identity matrix,
\item [(2)] $A_i^t=A_i,$ for $0\le i \le d,$
\item [(3)] $A_iA_j=\sum_{k=0}^d p_{ij}^k A_k$ for some integers $p_{ij}^k$ (the intersection numbers).
\end{enumerate}
In the special case that $(X,R_1)$ is a distance regular graph, and the $R_i$'s are the distance relation in that graph, we have with the notation of \S \ref{sec2.1} the relations $b_i=p_{1,i}^i,\,c_p=p_{1,i+1}^i.$
The algebra spanned over the complex by the $A_i$'s
is called the {\em Bose-Mesner} algebra. It admits as primitive idempotents $E_i$ of rank $m_i$ for $i=0,\dots d.$ 
The {\em  Krein parameters} are then defined by the expansion
$$ E_i\circ E_j=\sum_{k=0}^d q_{ij}^k E_k,$$
where $\circ$ stands for entrywise product.
A scheme is {\em Q-polynomial} iff $q_{ij}^k=0$ holds whenever the triple of integers $(i,j,k)$ violates the triangle inequality.
Let $b_i^*=q_{1,i}^i,\,c_i^*=q_{1,i+1}^i.$
The {\em Krein array} is then defined as  $$\{b_0^*,\dots,b_{d-1}^*;c_0^*,\dots,c_{d-1}^*\}.$$

The $P-Q$ duality is then illustrated in the following table.

\begin{table}[h]
\renewcommand{\arraystretch}{2}
\setlength\tabcolsep{7pt}
\centering
\begin{tabular}{|c|c|c|c|c|}
\hline
P&$v_i$&$p_{ij}^k$&$b_i$&$c_i$\\
\hline
Q&$m_i$&$q_{ij}^k$&$b_i^*$&$c_i^*$\\
\hline
\end{tabular}
\end{table}
\section{General graphs}\label{sec3}
In this paragraph we consider a graph $\Gamma$ that may not be regular, and a distinguished vertex $x$ of that graph. Let $v_i=|\Gamma_i(x)|.$
{\thm \label{ex} There exists an infinite family of graphs of diameter $2$ with a vertex $x$ such that $v_1=2,$ and $v_2$ unbounded.}

\begin{proof}
Consider the graph formed by adding to the bipartite complete $K_{2\times n}$ a vertex $x$ and two edges connecting $x$ to the two vertices of the part of size $2$ of $K_{2\times n}.$ We have $v_1=2$ and
$v_2=n.$
\end{proof}

The example of Theorem \ref{ex} yields a non log-concave valency sequence $(1,2,n)$  as soon as $n>4.$ It can be extended easily to a sequence $(1,2,\dots,2,n)$ by inserting a path of even length in $K_{2\times n}$ instead of a path of length $2.$

 A natural candidate for more counterexamples is that of distance degree regular graphs \cite{DDR}, where the numbers $v_i$ defined here do not depend on the chosen $x.$ The so-called 
 TriangleReplacedPetersen Graph in Figure \ref{fig1} is DDR and has a non log-concave series of valencies since its starts $(1,3,4,6,\dots)$. And this graph is distance degree regular as it is vertex-transitive.\par
\begin{figure}[h]\label{fig1}
  \centering
  \includegraphics[height=4.5cm, width=9.5cm]{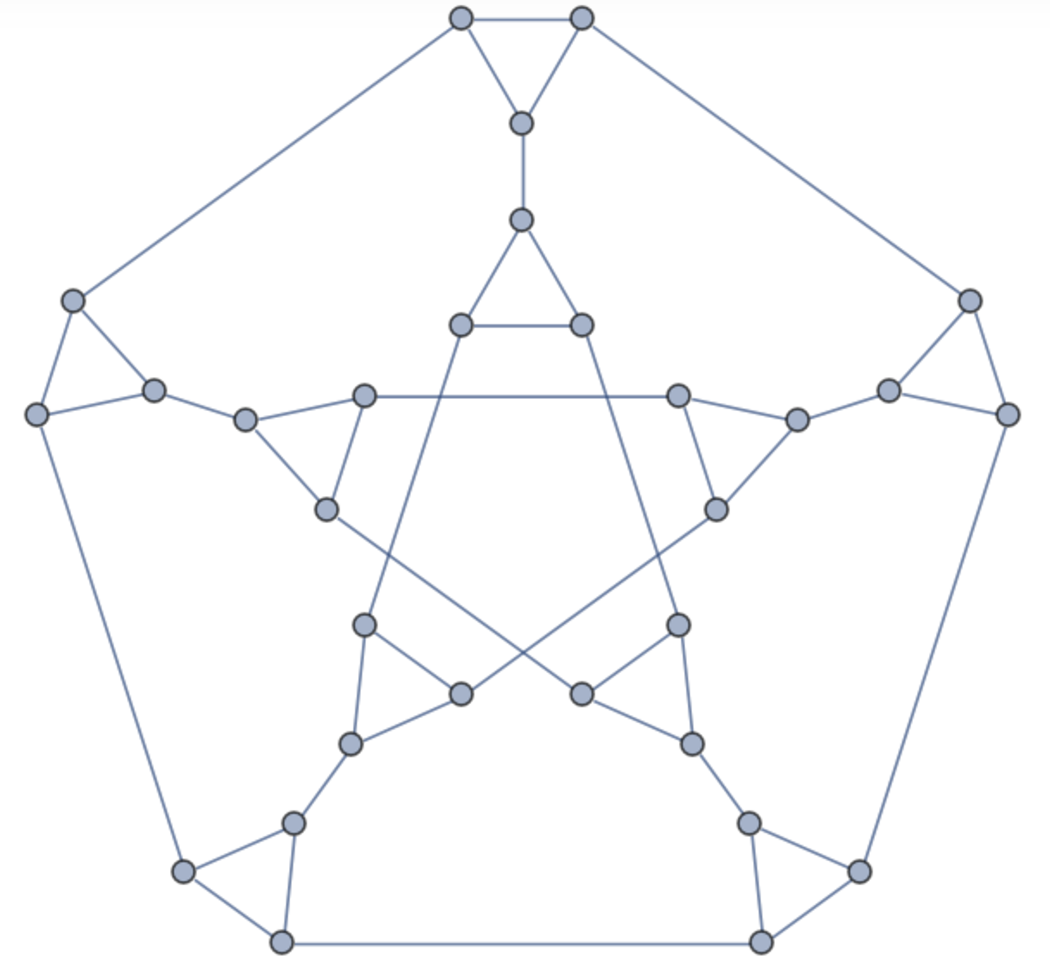}
  \caption{TriangleReplacedPetersen Graph}
  \label{fig1}
\end{figure}

 On a more positive side, we give a general construction of graphs where the valencies form a log-concave sequence.
Call $LC$ a graph $A$ with at least one vertex $x$ such that $a_i(x)=|A_i(x)|$ is log-concave. Then we have the following result.
\begin{thm}\label{LC} The Cartesian product of two $LC$ graphs is $LC$.  
\end{thm}
\begin{proof}Write $f_A(x;t)=\sum_{i=0}^d a_i(x)t ^i$. If $A$ and $B$ are $LC$ graphs then it can be shown that $f_{A\square B}((x_1,x_2);t)=f_A(x_1;t)f_B(x_2;t)$ by the definition of Cartesian product. Then the result follows by \cite[Prop. 2 ]{Sta}.
\end{proof}

This result can be strengthened if the graphs are DDR.

{\coro \label{DDR} The Cartesian product of two DDR $LC$ graphs is DDR $LC$.}

\begin{proof} In this case, because $A$ is DDR, we can write $f_A(x;t)=f_A(t),$
 a quantity independent of $x.$ The identity
 $f_{A\square B}((x_1,x_2);t)=f_A(t)f_B(t),$
 shows that $A\square B$ is DDR.
 If, furthermore, both $A$ and $B$ are LC, then 
 $A\square B$ is LC, by the same argument as in the proof of Theorem \ref{LC}.
\end{proof}

This yields a general construction.

{\coro \label{gen} If $A$ is DDR LC then $A^{\square n}$ is DDR LC for all integers $n.$}

\begin{proof}
 Proof by induction on $n,$ relying on Corollary \ref{DDR}, starting at $n=1.$
\end{proof}

\begin{ex}
We give some examples of constructions from Corollary \ref{gen}:
\begin{enumerate}
 \item[(1)] If $A=K_2,$ we obtain the graph of $H(n,2),$ also known as the hypercube, with valencies the binomial coefficients ${n \choose i}.$
 \item[(2)] If $A=K_q,$ for some integer $q,$ we obtain the graph of $H(n,q),$ the Hamming scheme, with valencies the binomial coefficients ${n \choose i}(q-1)^i.$
 \item[(3)] If $A=C_q,$ for some integer $q,$ we obtain
 the graph with geodetic distance the Lee metric on $\Z_q^n$  \cite{L}.
 There is no closed form expression for the $v_i$'s,
 as witness by the following generating functions \cite{B}:
 \noindent
 
  If $q=2s$ then $$f_{C_q^{\square n}}=(1+2\sum_{i=1}^{s-1} t^i+t^s)^n;$$
  if $q=2s+1$ then $$f_{C_q^{\square n}}=(1+2 \sum_{i=1}^s t^i)^n.$$

\end{enumerate}
\end{ex}
We give an asymptotic property that shows the power
of the Cartesian power construction.

{\thm For every graph $A$ there is an integer $n_0$
such that for $n>n_0,$ the graph $A^{\square n}$ is  LC.}

\begin{proof}
 Follows by the relation
 $$f_{A^{\square n}}(x;t)=f_A(x;t)^n,$$
 valid for any vertex of $A,$
 upon applying \cite[Theorem 1]{OR}.
\end{proof}
\section{Distance regular graphs}\label{sec4}
In this section, we focus on the class of distance regular graphs. The following result is contained in the proof of \cite[Prop. 1.4, p.205]{BI}.
{\thm (Bannai-Ito) \label{fonda} Let $\Gamma$ be a distance regular graph of diameter $d\ge 2$ with intersection array ${v_1,b_1,b_2,\dots,b_{d-1};1,c_2,\dots,c_d}.$ Then the sequence of its valencies $v_i$ is log-concave.}

\begin{proof}
By \cite[(2)]{TL} we know that $\frac{v_i}{v_{i-1}}=\frac{b_{i-1}}{c_i}.$ We need to check that $\frac{v_i}{v_{i-1}}\ge \frac{v_{i+1}}{v_i},$ which is therefore equivalent to $\frac{b_{i-1}}{c_i}\ge \frac{b_{i}}{c_{i+1}}.$
By \cite[Prop. 1 ]{TL} we know that the $b_i$'s are decreasing, and that the $c_i$'s are increasing. This implies the previous inequality upon quotienting $b_{i-1}\ge b_{i},$ by ${c_i}\le c_{i+1}.$
\end{proof}

\begin{ex}
We obtain the following log-concave sequences from the Hamming and Johnson graphs. Their log-concavity is well-known from statistics literature \cite{AP}.
\begin{enumerate}
\item[(1)] the binomials $v_i={ n\choose i}$ from $\Gamma$ the graph of $H(n,2),$ also known as the hypercube,
\item[(2)] the weighted binomials $v_i={ n\choose i}(q-1)^i$ from $\Gamma$ the graph of $H(n,q),$
\item[(3)] the hypergeometric distribution $v_i={ d\choose i}{ n-d\choose i}$ from $\Gamma$ the Johnson graph $J(n,d)$. 
\end{enumerate}

Other DR graphs yield log-concave sequences that are not so well-known. We only give two examples.

\begin{enumerate}
\item[(4)] the valencies $v_j =q^{j(j-1)} \frac{ \prod_{i=0}^{2j-1}(q^{m-i}-1) }{\prod_{i=1}^{j}(q^{2i}-1)}$ of symplectic form scheme $S(m,q)$ \cite[p.282]{BCN},
\item[(5)] the valencies $v_j= {d \brack d-j}{e \brack e-j}\prod_{i=0}^{j-1}(q^j-q^i)$ of the bilinear form scheme $B(d,e,q)$  \cite[p.280]{BCN}, where ${a \brack b}$ denotes a Gaussian binomial coefficient to the base $q.$
\end{enumerate}
\end{ex}

\begin{rem}
There are obvious generalizations of  Theorem \ref{fonda} to graphs that admit intersection arrays centered in a point. The most famous example is that of distance biregular graphs; bipartite graphs that admit two intersection arrays, the same on each vertex of the same color class \cite{Del}. However, many generalizations of distance regular graphs are not based on intersection arrays, but on polynomials in the adjacency matrix \cite{A}.
\end{rem}

The case $d=2$ of  Theorem \ref{fonda}  is worth of interest in its own right.

{\coro If $G$ is a $(v,k,\lambda,\mu)$ SRG then $$v \ge k+1+\lceil \sqrt{k}\rceil.$$}

\begin{proof}
The inequality $v\le k^2+k+1$ follows by Theorem \ref{fonda} with $v_1=k,$ and $v_2=v-k-1.$
The same method applied to the complementary graph yields $(v-k-1)^2\ge k.$ The inequality follows.
\end{proof}

\begin{rem}The upper bound in the proof is trivial because the so-called Moore bound \cite{HS} implies an upper bound for the number of vertices, that is 
$$v \le 1+k+k(k-1)=k^2+1$$ for $d=2$.
\end{rem}

\begin{ex}
If $\Gamma$ is the pentagon $C_5,$
an SRG on $5$ points of degree $k=2,$ we get
$$ 5\le v\le 5. $$
The lower bound and the Moore bound are met in this case.
\end{ex}
 
{\coro If $C$ is a projective two-weight code over a finite field with weights $w_1$ and $w_2$ then its weight distribution satisfies
$A_{w_1}^2\ge A_{w_2},$ and $A_{w_2}^2\ge A_{w_1}.$}

\begin{proof}
Following Delsarte \cite{D} consider the SRG with vertex set $C$ where two codewords are connected iff they differ by a codeword of weight $w_1.$ We apply Theorem \ref{fonda} with $v_i=A_{w_i},$ for $i=1,2.$
The first inequality follows. The second inequality follows in the same way by considering the complementary graph.
\end{proof}

\begin{rem} We cannot remove the word "projective" from the statement as the next example shows. The binary code obtained as the row span over $\F_2$ of the matrix $G$ below is a $[4,2;{2,3}]_2$ code with weight distribution $(1,1,2).$
$$G=\begin{pmatrix} 1 & 1 &0 & 0\\
0 & 1& 1 & 1
\end{pmatrix}$$
It is not projective as the third and fourth column of $G$  are equal.
\end{rem}

Many examples of low diameter DR graphs can be constructed from completely regular codes \cite{BCN,SS}. These are defined as codes the weight distribution of each coset solely depends on the coset weight.

{\coro Let $C$ be a projective completely regular linear code in $H(n,q).$ Denote by $d_i$ the number of cosets of weight $i.$ Then the $d_i$'s form a log-concave sequence.}

\begin{proof}
As is well-known \cite{BCN,SS}, the coset graph of $C$ is DR of diameter the covering radius of $C$. The valencies of this graph are  the $d_i$ 's. The result follows then by Theorem \ref{fonda}.
\end{proof}

The analogue of  Theorem \ref{fonda} for the $P-Q$ duality described in \S \ref{sec2.3} is given here conditioned on a property of the scheme. Let us
say that a Q-polynomial scheme has Property $M$ (for monotone) if
for $0 \le i \le d-1,$ we have both
\begin{itemize}
 \item 
$ b_i^*\ge b_{i+1}^*, $
\item $ c_i^*\le c_{i+1}^*. $
\end{itemize}

It seems that most Q-polynomial schemes enjoy Property $M$ except certain Johnson schemes \cite{S}.\par

{\thm If $(X,R)$ is a Q-polynomial  association scheme having Property $M$ then its multiplicities $(m_i)$ form a log-concave sequence.}

\begin{proof}
 By \cite[Prop. 3.7, p.67]{BI}, or \cite[Lemma 2.3.1 p.49]{BCN} we know that $\frac{m_i}{m_{i-1}}=\frac{b_{i-1}^*}{c_i^*}.$
 By Property $M$, we complete the proof 
 in a way similar to the proof of Theorem \ref{fonda}.
\end{proof}

\begin{ex}By \cite{WI}, the sequence formed by the multiplicities of Johnson scheme $J(21,3)$ is given by $\{ 20,833/57,11200/1377;1,2380/1539,70/51\}$. The sequence doesn't satisfy Property $M$ as $2380/1539\le70/51$. However, the multiplicities form a log-concave sequence, as the next result shows. 
\end{ex}
A sufficient condition on $(n,d)$ for  the multiplicities of the Johnson scheme $J(n,d)$ to form a log-concave sequence is shown below.
\begin{thm}\label{thm4}For {$d< \frac{{{{n + 1}}}}{2}$}, the sequence formed by multiplicities of Johnson scheme $J(n,d)$ is log-concave.
\end{thm}
\begin{proof}According to \cite{DA}, the multiplicity $m_i$ of $J(n,d)$ is given below
\[{m_i} = \left( {\begin{array}{*{20}{c}}
n\\
i
\end{array}} \right) - \left( {\begin{array}{*{20}{c}}
n\\
{i - 1}
\end{array}} \right) = \frac{{n - 2i + 1}}{{n - i + 1}}\left( {\begin{array}{*{20}{c}}
n\\
i
\end{array}} \right).\]
Note that the sequence $(a_{i}b_i)$ is log-concave if the corresponding two sequences $(a_{i})$ and $(b_i)$ are log-concave. Note that ${\left( {\begin{array}{*{20}{c}}
n\\
i
\end{array}} \right)^2} \ge \left( {\begin{array}{*{20}{c}}
n\\
{i - 1}
\end{array}} \right)\left( {\begin{array}{*{20}{c}}
n\\
{i + 1}
\end{array}} \right)$ for $i = 2, \cdots ,d-1,$
directly or as an application of Theorem \ref{fonda}. Let $m\left( {n,i} \right) = \frac{{n - 2i + 1}}{{n - i + 1}}$, we can deduce that \[\frac{{m{{(n,i)}^2}}}{{m(n,i - 1)m(n,i + 1)}} - 1 = \frac{{3(n + 1)({{n  - }}\frac{4}{3}{{i  +  1}})}}{{{{(n  -  2i  -  1)(n  -  2i  +  3)(n  -  i  +  1}}{{{)}}^2}}}\ge 0 \] as $1< i\le {d-1< \frac{{{{n - 1}}}}{2}}$. Then the sequence $(m(n,i))$ is log-concave with respect to $i$ and the result follows.
\end{proof}
{By \cite{BCN}, several graphs related to Johnson graphs also have similar results.}
{\begin{coro}If $m \ge 4$ and $d=[\frac{m}{2}]$, the sequence formed by multiplicities of the folded Johnson graph $\overline{J}(2m,m)$ is log-concave.
\end{coro}
\begin{proof}The multiplicity $m_i$ is given by \[{m_i} = \left( {\begin{array}{*{20}{c}}
2m\\
2i
\end{array}} \right) - \left( {\begin{array}{*{20}{c}}
2m\\
{2i - 1}
\end{array}} \right) = \frac{{2m - 4i + 1}}{{2m - 2i + 1}}\left( {\begin{array}{*{20}{c}}
2m\\
2i
\end{array}} \right).\] Note that the sequence $(\frac{{2m - 4i + 1}}{{2m - 2i + 1}})$  is log-concave if $d=[\frac{m}{2}]$ by the computation in the proof of Theorem \ref{thm4}. Let $v_i={ 2m\choose 2i}$, then $(v_i)$ is log-concave since \[\frac{{v_i^2}}{{{v_{i - 1}}{v_{i + 1}}}} = \frac{{(2m - 2i + 2)(2m - 2i + 1)(2i + 2)(2i + 1)}}{{(2m - 2i)(2m - 2i - 1)(2i)(2i - 1)}} > 1.\]
Then the result follows.
\end{proof}}

{\begin{coro}The sequence formed by multiplicities of the Odd graph $\mathcal{O}_{m+1}$ is log-concave.
\end{coro}
\begin{proof}According to \cite{BCN}, $d=m$ and $m_i={ 2m\choose i}-{ 2m\choose i-1}$. By Theorem \ref{thm4}, ${m_{i}}^2\ge{m_{i-1}}{m_{i+1}}$ when $1< i \le d-1<\frac{2m-1}{2}$. Then the result follows.
\end{proof}}

A sufficient condition on $(n,d)$ for  the multiplicities of the Grassman scheme $Gr(n,d)$ to form a log-concave sequence is shown below. This is a $q$-analogue of the previous result.
\begin{thm}For ${d < \frac{{{{n + 1}}}}{2}}$, the sequence formed by multiplicities of Grassman scheme $Gr(n,d)$ is log-concave.
\end{thm}
\begin{proof}
The multiplicities of the Grassman scheme are
$$m_i={ n\brack i}-{ n\brack i-1}=(1-\frac{q^{i}-1}{q^{n-i+1}-1}){ n\brack i},$$
where ${ \brack }$ denotes the Gaussian binomial to the base $q.$
 According to \cite[p.501]{Sta} the sequence formed by Gaussian binomials is log-concave.
 By the same argument as in the proof of the previous theorem
 we are reduced to check that
 $u_j=1-\frac{q^{j}-1}{q^{n-j+1}-1}$ is log-concave.
 A computation in Wolfram alpha shows that
 $$\frac{u_j^2}{u_{j-1}u_{j+1}}-1=\frac{(q - 1)^2 q^{j+n}A}{(q^j- q^{n+1})^2 (-q^{2j + n} - q^{2j + n + 4} + q^{4j+1} + q^{2 n + 3})},$$
where $A=q^{j+2n+2}+2q^{j+2n+3}+q^{j+2n+4}-2q^{2j+n+1}-2q^{2j+n+2}-2q^{2j+n+3}+2q^{3j+1}+q^{3j+2}-q^{4j+1}+q^{3j}-q^{2n+3}$. 
As {$2j\le 2(d-1)<n-1$}, we have
\[-q^{2j + n} - q^{2j + n + 4} + q^{4j+1} + q^{2 n + 3}=(q^{n+3}-q^{2j})(q^{n}-q^{2j+1})>0.\]
Note that the sign of $\frac{u_j^2}{u_{j-1}u_{j+1}}-1$ is the same as the sign of $A$.
As $q \ge 2$ and $1< j < n$, it's easy to check that $2{q^{2j + n + 1}} \le {q^{2j + n + 2}} < {q^{j + 2n + 2}}$. Then for $1< j < \frac{{{{n - 1}}}}{2}$, we have
\[A > {q^{2n + 3}}({q^j} - 1) - {q^{4j + 1}} + 2{q^{3j + 1}} + {q^{3j}}>0. \]
Therefore, $(u_j)$ is log-concave when $1<j\le d-1<\frac{n-1}{2}$.
\end{proof}

\begin{rem}
For the results obtained in this section, we give the following remarks:
\begin{enumerate}
 \item[(1)] Since both $H(n,q)$ and the scheme of bilinear forms are self-dual, as translation schemes \cite[\S 2.10]{BCN}, we recover points (1),(2) and (5) of Example 2 above.
 
 \item[(2)] The unimodality of the multiplicities of
 a Q-polynomial association scheme has been an open problem since the 1980's \cite[p.205]{BI}.
 \item[(3)] It would be a worthy research project to characterize $Q$-polynomial schemes enjoying Property $M$.
\end{enumerate}
\end{rem}

\section{Conclusion}\label{sec5}
In this note, we have shown that the series of valencies of a distance regular graph, and the series of multiplicities of certain $Q$-polynomial association schemes are log-concave. The first result is well-known since \cite[Proposition 1.4, p.205]{BI}, but the consequences on strongly regular graphs, and two weight codes are worth exploring. The second result approaches a problem open since the 1980's \cite[Conjecture 1, p.205]{BI} .
The log-concavity of the valencies cannot be obtained for general graphs as Section \ref{sec3} shows, even in the restricted class of distance degree regular graphs \cite{DDR}. Therefore, it is a challenging open problem to find a larger class of graphs where they could be extended.
Similarly, it would be interesting to find a larger class of codes than completely regular codes for which the series of coset weights is log-concave.\\

\noindent {\bf Acknowledgement}\\\\ The third author thanks
Sacha Gavrilyuk, Akihiro Munemasa, and Sho Suda for helpful discussions.\\

\noindent {\bf Data availability} \\\\Any data generated is available upon request from the authors.\\

\noindent {\bf Declaration} \\\\No AI program was used to write this manuscript.\\

\end{document}